\documentclass[11pt,reqno]{amsart}
\usepackage{amssymb,amsmath,amscd}
\usepackage{latexsym,bm,bbm,mathrsfs}
\usepackage{hyperref,graphicx, enumerate}
\usepackage{mathtools}
\usepackage{stmaryrd}
\usepackage[all,pdf]{xy}
\usepackage{graphicx}
\usepackage{tikz-cd}
\usepackage{color}
\graphicspath{ {./} }

\usepackage{stackrel}
\usepackage[left=30mm, right=30mm, top=30mm, bottom=30mm]{geometry}
\usepackage[shortlabels]{enumitem}
\usepackage{ulem}
\usepackage{verbatim}

\allowdisplaybreaks

\newtheorem{thm}{Theorem}[section]
\newtheorem{quest}{Question}
\newtheorem*{quest*}{Question}

\newtheorem{theorem}[thm]{Theorem}
\newtheorem*{theorem*}{Theorem}
\newtheorem{cor}[thm]{Corollary}
\newtheorem{prop}[thm]{Proposition}
\newtheorem{lemma}[thm]{Lemma}

\newtheorem{defn}[thm]{Definition}
\newtheorem*{defn*}{Definition}

\theoremstyle{definition}
\newtheorem{remark}[thm]{Remark}
\newtheorem*{remark*}{Remark}

\newcommand{\calB}{{\mathcal{B}}}

\newcommand{\bbq}{{\mathbb{Q}}}

\newcommand{\bbz}{{\mathbb{Z}}}
\newcommand{\bbf}{\mathbb{F}}

\newcommand{\Gal}{\operatorname{Gal}}

\newcommand{\tors}{\operatorname{tors}}

\newcommand{\GL}[1]{{\operatorname{GL}_{2}\left({#1}\right)}}

\newcommand{\GLl}{{\operatorname{GL}_{2}\left(\mathbb{F}_{\ell}\right)}}
\newcommand{\SLl}{{\operatorname{SL}_{2}\left(\mathbb{F}_{\ell}\right)}}
\newcommand{\Fl}{\mathbb{F}_{\ell}}
\newcommand{\Fll}{\mathbb{F}_{\ell^{2}}}
\newcommand{\Flx}{\mathbb{F}_{\ell}^{\times}}
\newcommand{\Fllx}{\mathbb{F}_{\ell^{2}}^{\times}}
\newcommand{\Bl}{\mathscr{B}}
\newcommand{\Cs}{\mathscr{C}_{s}}
\newcommand{\Cns}{\mathscr{C}_{ns}}
\newcommand{\Ns}{\mathscr{N}_{s}}
\newcommand{\Nns}{\mathscr{N}_{ns}}
\newcommand{\diagexp}[1]{\operatorname{diagexp}\left({#1}\right)}
\newcommand{\diag}[1]{\operatorname{diag}\left({#1}\right)}

\newcommand{\calK}{\mathcal{K}}
\newcommand{\calG}{\mathcal{G}}
\newcommand{\calP}{\mathcal{P}}
\newcommand{\PDI}{\mathcal{PDI}}
\newcommand{\calR}{\mathcal{R}}

\newcommand{\cyc}{\operatorname{cyc}_{\ell}}

\usepackage{enumitem}

\title[Growth of Torsion Groups of Elliptic Curves]{Growth of torsion groups of elliptic curves over number fields  without rationally defined CM}

\author{Bo-Hae Im}
\address{Department of Mathematical Sciences, KAIST, 291 Daehak-ro, Yuseong-gu, Daejeon, 34141, South Korea}
\email{bhim@kaist.ac.kr}
\thanks{Bo-Hae Im was supported by Basic Science Research Program through the National Research Foundation of Korea (NRF) funded by grant funded by the Korea government (MSIT) (NRF-2023R1A2C1002385).}

\author{Hansol Kim}
\address{Department of Mathematical Sciences, KAIST, 291 Daehak-ro, Yuseong-gu, Daejeon, 34141, South Korea}
\email{jawlang@kaist.ac.kr}

\date{\today}

\begin{document}

\maketitle

\begin{abstract}
	For a quadratic field $\calK$ without rationally defined complex multiplication, we prove that there exists of a prime $p_{\calK}$ depending only on $\calK$ such that 
	if $d$ is a positive integer whose minimal prime divisor is greater than $p_{\calK}$, then for any extension $L/\calK$ of degree $d$ and any elliptic curve $E/\calK$, we have $E\left(L\right)_{\tors} = E\left(\calK\right)_{\tors}$. By not assuming the GRH, this is  a generalization of  the results by Genao, and Gon{\'a}lez-Jim{\'e}nez and Najman.
\end{abstract}

\section{Introduction}\label{intro}

For an elliptic curve $E$ defined over $K$, it is well known that the Mordell-Weil group  $E\left(K\right)$ of rational points of $E$ is a finitely generated abelian group (\cite[VIII.Theorem~6.7]{Sil}), i.e., $E\left(K\right)$ is isomorphic to $E\left(K\right)_{\tors} \oplus \bbz^{r_{E\left(K\right)}}$, where  $E\left(K\right)_{\tors}$ is a finite abelian group and $r_{E\left(K\right)}$ is a non-negative integer. We call $E\left(K\right)_{\tors}$ and $r_{E\left(K\right)}$ by the torsion part and the rank of $E$ over $K$, respectively. Extensive research has been conducted to determine an efficient approach for computing the rank and torsion subgroup of an elliptic curve over a number field, including the case of $\bbq$. However, no efficient method has been discovered thus far.

For the torsion parts of elliptic curves, the introduction in \cite{Gu21} provides a comprehensive overview of the relevant prior results and by referring to it,  we include here some of those that are of our interest: For an elliptic curve $E$ over $\bbq$, the torsion subgroup ${E\left(\bbq\right)}_{\tors}$ has been completely classified by Mazur (\cite[Theorem~2]{M78}). For a number field $K$ and an elliptic curve $E$ over $K$, if $L$ is an extension over $K$ of degree $4$ or of prime degree, then ${E\left(L\right)}_{\tors} $ is determined (\cite{KM88}, \cite{Najman16}, \cite{Chou16}, \cite{GJ17}, \cite{GJN20}).

There have been multiple approaches developed for studying the torsion parts of elliptic curves over $K$, and one of them is to study the structure of the Galois group  $\Gal\left(K\left(E\left[N\right]\right)/K\right)$ of the field of definition of $N$-torsion points of an elliptic curve $E/K$ over a number field $K$ via the $2$-dimensional mod-$N$ Galois representation of $\Gal\left(K\left(E\left[N\right]\right)/K\right)$. 
Serre has proved in \cite{Serre72} that for any number field $K$ and any non-CM elliptic curve $E/K$, there is a constant $C$ depending on $K$ and $E$ such that for any prime $\ell > C$, $\Gal\left(K\left(E\left[N\right]\right)/K\right)$ is isomorphic the $\GLl$. Moreover, Serre conjectured that the constant $C$ may depend only on $K$. Moreover,  for an elliptic curve $E/\bbq$, the $2$-dimensional mod-$\ell$ Galois representations of $\Gal(\bbq(E\left[\ell\right])/\bbq)$ are computed by Zywina in \cite[Sections~1.8, 1.9]{Z15} with their MAGMA codes posted in his homepage and in this result, it has been seen that for CM elliptic curves over $\bbq$, its non-trivial endomorphism algebra (over $\overline{\bbq}$) plays a crucial role in characterizing the structures of the Galois group. See \cite[Sections~7.1, 7.2]{Z15}.

Another approach to study the torsion parts of elliptic curves over $K$ is via isogeny. For any elliptic curve $E/K$ and any finite subgroup $V$ of $E\left(\overline{K}\right)$, there is a unique isogeny $\alpha: E \to E'$ defined over $\overline{K}$  with kernel $V$. Moreover, $V$ is $\Gal\left(\overline{K}/K\right)$-invariant if and only if $\alpha$ is $K$-rational (\cite[Exercise~3.13e]{Sil}). Therefore, we may identify $V$ with $\alpha$. Moreover, a $\Gal\left(\overline{K}/K\right)$-invariant finite subgroup of $E\left(\overline{K}\right)$ corresponds to a $K$-rational isogeny from $E$. Hence, for positive integers $m$ and $n$ such that $m \mid n$, the existence of a $K$-rational isogenies whose kernel is isomorphic to $\bbz/m\bbz \times \bbz/n\bbz$ is weaker than the existence of an elliptic curve over $K$ whose torsion part containing $\bbz/m\bbz \times \bbz/n\bbz$. Any $K$-rational isogeny $\alpha: E \to E'$ whose kernel is isomorphic to $\bbz/m\bbz \times \bbz/n\bbz$ factors as $\beta \circ \left[ m\right]$ for a unqiue $K$-rational isogeny $\beta: E \to E'$ and $\ker \beta$ is isomorphic to $\bbz/\frac{n}{m}\bbz$. An isogeny with the cyclic kernel is called {\it a cyclic isogeny}. Obviously, any isogeny of prime degree is cyclic. If $V$ is $\Gal\left(\overline{K}/K\right)$-invariant and cyclic, then so are all subgroups of $V$ since they are of form $\left\{ P\in V: \left[n\right] P = O \right\}$. Hence, any $K$-rational cyclic isogeny is a composition of $K$-rational cyclic isogenies of prime degree. So the torsion subgroups of elliptic curves $E/K$ can be studied by investigating whether there is a $K$-rational isogeny of prime degree. Over $\bbq$, Mazur proves \cite[Theorem~3]{M78} which proves that there are $12$ prime degrees of $\bbq$-rational isogenies.

For any finite extension $L$ of $K$, it is obvious that $E\left(L\right)_{\tors} \supseteq E\left(K\right)_{\tors}$ and $r_{E\left(L\right)} \ge r_{E\left(K\right)}$. So we might expect that the torsion parts and the ranks of elliptic curves might grow upon base change. It is  natural to ask which finite extensions $L/K$ preserve the torsion parts of elliptic curves $E/K$ without any growth upon base change from $K$ to $L$. Regarding this question, Gon{\'a}lez-Jim{\'e}nez and Najman (\cite{GJN20}) gave a partial answer for elliptic curves over $\bbq$ as the follows:
\begin{theorem*}[{\cite[Theorem~7.2.i]{GJN20}}]
	Let $d$ be a positive integer whose minimal prime divisor is greater than $7$. For any extension $L/\bbq$ of degree $d$ and any elliptic curve $E/\bbq$, $ E\left(L\right)_{\tors} = E\left(\bbq\right)_{\tors}$.
\end{theorem*}

Motivated by {\cite[Theorem~7.2.i]{GJN20}}, the following question captures our primary attention and focus.

\begin{quest}\label{quest_org}
	Let
 $K$ be a number field. Does $K$ satisfy the following  $\mathbf{Property}~\calP\left(K\right)$?
\end{quest}
\begin{equation*}
	\begin{aligned}
	\mathbf{Property}~\calP\left(K\right):
		& \text{ There exists a prime } p_{K} \text{ depending only on } K \text{ such that any elliptic curve}\\
		& E/K \text{ satisfies } E\left(L\right)_{\tors} = E\left(K\right)_{\tors} \text{ for any extension } L/K \text{ of degree whose}  \\
		& \text{minimal prime divisor is greater than } p_{K}\text{.}
    \end{aligned}
\end{equation*}

A partial answer to Question~\ref{quest_org} is given by Genao in \cite{G22}. To describe the results in \cite{G22}, we recall the definition of RCM which stands for rationally defined complex multiplication.

\begin{defn} [{\cite[p.2]{G22}}]
    We say that a number field $K$ has RCM (rationally defined complex multiplication)  if there exists a CM elliptic curve  defined over $K$ whose endomorphism ring is $K$-rational.
\end{defn}

It is pointed out in \cite[Theorem~3]{G22} that the answer to Question~\ref{quest_org} is affirmative only for $K$ without RCM. Moreover, assuming GRH (generalized Riemann hypothesis), \cite[Theorem~1]{G22} states that the answer to Question~\ref{quest_org} is  affirmative if and only if $K$ has no RCM.

Not assuming GRH, we prove that the answer to Question~\ref{quest_org} is affirmative for quadratic fields without RCM as a generalization of the results by Genao(\cite{G22}) and Gon{\'a}lez-Jim{\'e}nez and Najman (\cite{GJN20}). We call a quadratic extension of $\bbq$    {\it a quadratic field} simply.

\begin{theorem}\label{quad_main}
	Every quadratic field $\calK$ without RCM satisfies $\mathbf{Property}~\calP\left(\calK\right)$.
\end{theorem}

The main ingredient of the proofs of our main results is  Proposition~\ref{woRCM_main} below  which gives the structure of the Galois group $\Gal\left( K\left(E\left[\ell\right] \right) / K\right)$ for a prime $\ell$ under certain assumptions without RCM, whose proof is given in Section~\ref{large}. 
We might expect from \cite[Theorem~7.2.i]{GJN20} that if $\mathbf{Property}~\calP\left(K\right)$ is satisfied, then such a prime  $p_{K}$ is at least $7$. Therefore, we are directed to study the Galois groups $\Gal\left( K\left(E\left[\ell\right] \right) / K\right)$ for primes $\ell$ and elliptic curves $E$ over $K$ with a point $R \in E\left[ \ell\right] - \left\{ O\right\}$ such that the degree $ \nmid \left[ K\left(R\right) : K\right]$ is not divisible by either the small primes $2$ or $3$. First, we analyze it independently of assuming the RCM, and we give a less robust result than Proposition~\ref{woRCM_main} for the purpose.
\begin{prop}\label{pre_main}
	Let $K$ be a number field. There exists a positive integer $N'_{K}$ depending only on $K$ such that  for every prime $\ell > N'_{K}$ and elliptic curve $E/K$ with a point $R\in E\left[\ell\right]-\left\{O\right\}$ such that $2,3\nmid \left[K\left(R\right):K\right]$, the Galois group $\Gal\left( K\left(E\left[\ell\right]\right)/K\right)$ is isomorphic to a subgroup of the Borel subgroup of $\GLl$.
\end{prop}

Proposition~\ref{pre_main} can be put in another way as follows: Under its assumption, $E\left[\ell\right]$ has a Galois invariant $1$-dimensional $\Fl$-subspace $V$ and the pair $\left(E,V\right)$ is a non-cuspidal $K$-raional point of the modular curve $X_{0}\left(\ell\right)$. We may identify $\left(E,V\right)$ to the $K$-rational isogeny $\alpha$ from $E$ with kernel $V$ in the sense as in Section~\ref{intro}. Hence, if we let  $Y_{0}\left(\ell\right)$ be the set of all non-cuspidal points in $X_{0}\left(\ell\right)$, then $\ell$ is  a {\bf p}rime {\bf d}egree of an {\bf i}sogeny  if and only if $Y_{0}\left(\ell\right)\left(K\right) \ne \emptyset$.
Let \begin{align}\label{setP}
	\PDI \left(K\right) := \left\{ \text{a prime } \ell : Y_{0}\left(\ell\right)\left(K\right) \ne \emptyset\right\}.
\end{align}
\
Then, Question~\ref{quest_org} is related to the following question.

\begin{quest}\label{quest_famous}
	For each number field $K$, is the set $\PDI \left(K\right)$ finite?
\end{quest}

Momose gave a sufficient condition to the finiteness of the set  $\PDI \left(K\right)$ of this question as follows.

\begin{theorem*}[{\cite[Theorem~B]{M95}}]\label{M95B}
	Let $\calK$ be a quadratic field. If $\calK$ is not an imaginary quadratic field of class number $1$, then $\PDI \left(\calK\right)$ is finite.
\end{theorem*}

\begin{remark}
	Let $K$ be a number field. It is well known that $\PDI\left(K\right)$ is infinite if $K$ contains the Hilbert class field of an imaginary quadratic field. Obviously, an imaginary quadratic field has class number $1$ if and only if it is the Hilbert class field of itself. This fact points out why the condition that the base field has no RCM is crucial, again. Conversely, under GRH, $\PDI\left(K\right)$ is infinite if and only if $K$ contains the Hilbert class field of an imaginary quadratic field (see \cite[Remark~8]{M95}).
\end{remark}

\begin{remark}\label{PDI-P}
	The proofs of \cite[Theorem~1]{G22} and \cite[Theorem~3]{G22} depend on the finiteness of $\PDI\left(K\right)$. For a number field $K$ without RCM, under the GRH, if $\PDI\left(K\right)$ is finite, then $\mathbf{Property}~\calP\left(K\right)$ holds. However, $\mathbf{Property}~\calP\left(K\right)$ may not hold for $K$ with RCM in which case  $\PDI\left(K\right)$ is infinite.
\end{remark}

Moreover, for a number field $K$, Momose classified in \cite{M95} all $K$-rational isogenies of sufficiently large prime degrees in $\PDI\left(K\right)$ into three types. If $K$ does not have RCM, only two types can occur as follows.

\begin{theorem}[{\cite[Theorem~A]{M95}}]\label{M95A}
	Let $K$ be a number field without RCM. For a prime $\ell$ and $\left(E,V\right) \in Y_{0}\left(\ell\right)\left(K\right)$, the natural representation $\lambda: \Gal\left(K\left(E\left[\ell\right]\right)/K\right) \to \operatorname{GL}\left(V\right) \cong \Flx$ is called the isogeny character of $\left(E,V\right)$. There is a constant prime number $C_{K}$ dpending only on $K$ such that any $\ell > C_{K}$ is one of the following types:

  \begin{enumerate}[{label=\textbf{Type \arabic*. }, leftmargin=5\parindent}]
		\item $\lambda^{12}$ or $\left(\cyc^{-1}\lambda\right)^{12}$ is unramified.
		\item $\lambda^{12} = \cyc^{6}$ and $\ell \equiv 3 \pmod{4}$ where $\cyc$ is the $\ell$-cyclotomic character.
	\end{enumerate}
\end{theorem}

\begin{defn}
	For a number field $K$ without RCM, we denote  by $\PDI_{2}\left(K\right)$ the set of all primes $\ell \equiv 3 \pmod{4}$ such that there is a $K$-rational isogeny of degree $\ell$ whose isogeny character $\lambda$ satisfies that $\lambda^{12} = \cyc^{6}$.
\end{defn}

For a number field $K$ without RCM, we can elaborate on Proposition~\ref{pre_main} to provide more precise properties.
\begin{prop}\label{woRCM_main}
    Let $K$ be a number field without RCM, there exists a positive integer $N_{K}$ depending only on $K$ such that the following holds: for every prime $\ell > N_{K}$ and an elliptic curve $E/K$, if there exists  $R\in E\left[\ell\right]-\left\{O\right\}$ such that $2,3\nmid \left[K\left(R\right):K\right]$, then we have the following; \begin{enumerate}[\normalfont (a)]
		\item the Galois group $\Gal\left( K\left(E\left[\ell\right]\right)/K\right)$ is isomorphic to a subgroup of the Borel subgroup of $\GLl$ containing an unipotent element,
		\item there exists a prime factor of  $\ell-1$ which divides  $\left[ K\left(S\right) : K\right]$, for all $S\in E\left[\ell\right]-\left\{O\right\}$ and
		\item $\ell \equiv 7,11,23,31,35 \pmod{36}$ and $\ell \in \PDI_{2}\left(K\right)$.
	\end{enumerate}
\end{prop}
\begin{proof}
	It follows from Proposition~\ref{Bl} since except for finitely many primes $\ell$ satisfy all three conditions (i)-(iii) in Proposition~\ref{Bl}.
\end{proof}

Then, by applying Proposition~\ref{woRCM_main}, we prove the following theorem whose proof is given in Section~\ref{main_proof}.

\begin{theorem}\label{main}
	Let $K$ be a number field without RCM and \begin{align}	\label{setR}
		\calR\left(K\right):= \left\{ \text{prime divisors of }\ell-1: \begin{aligned}&\ell \in \PDI_{2}\left(K\right)\text{ in \eqref{setP}},\\& \ell \equiv 7,11,23,31,35\pmod{36} \end{aligned}\right\}.
	\end{align}
	If $\calR\left(K\right)$ is finite,  then $K$ satisfies $\mathbf{Property}~\calP\left(K\right)$.
\end{theorem}

Finally, our main result, Theorem~\ref{quad_main} is proved by applying Theorem~\ref{main} and the following result \cite[Theorem~4]{M95} by Momose and we give its proof in Section~\ref{main_proof}.

\begin{theorem*}[{\cite[Theorem~4]{M95}}]\label{M95T4}
	For any quadratic field $\calK$, $\PDI_{2}\left(\calK\right)$ is finite.
\end{theorem*}

\begin{remark}\label{PDI-PDI2}
	The proofs of 
 Theorem~\ref{quad_main} and Theorem~\ref{main} rely on the finiteness of $\PDI_{2}\left(K\right)$, while  \cite[Theorem~1]{G22}  is proved under the assumption on finiteness of   $\PDI\left(K\right)$.
\end{remark}

\section{The $\ell^{\infty}$-torsion subgroups for sufficiently large primes $\ell$ : the proofs of Proposition~\ref{pre_main} and Proposition~\ref{woRCM_main}}\label{large}

In this section, we investigate the possible structures of subgroups of $\Gal\left(K\left(E\left[\ell\right]\right)/K\right)$ and we prove Proposition~\ref{pre_main} and Proposition~\ref{woRCM_main} by applying them.

\subsection{Subgroups of $\Gal\left(K\left(E\left[\ell\right]\right)/K\right)$ and their corresponding subgroups of $\GLl$}
Let $K$ be a number field and $E/K$ be an elliptic curve over $K$. For a positive integer $N$, recall that $E\left[N\right]\cong \bbz/N\bbz\times \bbz/N\bbz$. So  we consider the Galois group $\Gal\left(K\left(E\left[N\right]\right)/K\right)$ as a subgroup of GL$_2(\bbz/N\bbz)$ depending on a basis $\calB=\left\{P,Q\right\}$ for $E\left[N\right]$ via the following map.

\begin{defn}
	Let $K$ be a number field, $E/K$ an elliptic curve over $K$, and $N$ a positive integer. For  a given (ordered) basis $\calB=\left\{P,Q\right\}$ for $E\left[N\right]$, we have the injective group homomorphism,
\begin{align*}
		\rho_{\calB}: \Gal\left(K\left(E\left[N\right]\right)/K\right) &\to \operatorname{GL}_{2}\left(\bbz/N\bbz\right) \text{ defined by }\\
\sigma& \mapsto  \rho_{\calB}\left(\sigma\right): \left(\begin{matrix}P\\Q\end{matrix}\right) \mapsto \left(\begin{matrix}P^{\sigma}\\Q^{\sigma}\end{matrix}\right).
\end{align*}

Let  $G\left(\calB\right):=\rho_{\calB}\left(\Gal\left(K\left(E\left[N\right]\right)\right)/K\right)$, the image of $\rho_\calB$.
\end{defn}


For an odd prime $\ell$ and a given basis $\calB=\left\{P,Q\right\}$ for $E\left[\ell\right]$, we consider the right action of the image $G\left(\calB\right)$ on the row vectors in  $ \Fl^{2}$.

\begin{defn}
	Let  $G $ be a subgroup of $\GLl$.  For each non-zero row vector $\left(c\,d\right) \in \Fl^{2}$, we define the following subgroups of $G$:
	\begin{itemize}
		\item $G_{c\,d} := \left\{ A\in G: \left(c\,d\right)A = \left(c\,d\right)\right\}$.
		\item $G^{\circ} := G \cap \SLl$.
	\end{itemize}

We note that $G_{c\,d}^{\circ}  = G_{c\,d} \cap \SLl = G_{c\,d} \cap G^{\circ}$.
\end{defn}


\begin{lemma}\label{easy}
	Let $K$ be a number field and  $E/K$ be an elliptic curve over $K$. For an odd prime $\ell$, let  $\zeta_{\ell}$ be a primitive $\ell$th root of unity, and for a given basis $\calB=\left\{P,Q\right\}$ for $E\left[\ell\right]$, we let $G = G\left(\calB\right)$. Then, for any non-zero row vector $\left(c\,d\right) \in \Fl^{2}$, we have the following: \begin{enumerate}[\normalfont (a)]
		\item $G_{c\,d} = \rho_{\calB}\left(\Gal\left(K\left(E\left[\ell\right]\right)/K\left(\left[c\right]P+\left[d\right]Q\right)\right)\right)$, and $\left[ G: G_{c\,d}\right] = \left[ K\left(\left[c\right]P+\left[d\right]Q\right): K\right]$. 
		\item $G^{\circ} = \rho_{\calB}\left(\Gal\left(K\left(E\left[\ell\right]\right)/K\left(\zeta_{\ell}\right)\right)\right)$, $G / G^{\circ} \cong \Gal\left(K\left(\zeta_{\ell}\right)/K\right)$, and $\left[ G: G^{\circ}\right] = \left[ K\left(\zeta_{\ell}\right): K\right]$.
		\item $G_{c\,d}^{\circ} = \rho_{\calB}\left(\Gal\left(K\left(E\left[\ell\right]\right)/K\left(\zeta_{\ell},\left[c\right]P+\left[d\right]Q\right)\right)\right)$, and $\left[ G: G_{c\,d}^{\circ}\right] = \left[ K\left(\zeta_{\ell}, \left[c\right]P+\left[d\right]Q\right): K\right]$.
		\item $G_{c\,d}^{\circ}$ is trivial, or it  is conjugate to $\left\langle U \right\rangle$, where $U=\begin{psmallmatrix}1&1\\0&1\end{psmallmatrix}$. In particular,  $\left| G_{c\,d}^{\circ} \right| \bigm| \ell$.
		\item $G_{c\,d}=\begin{cases} G_{1\,d/c}, &\text{ if } c\ne 0,\\ G_{01}, &\text{ if } c=0.\end{cases}$
	\end{enumerate}


 \
 
	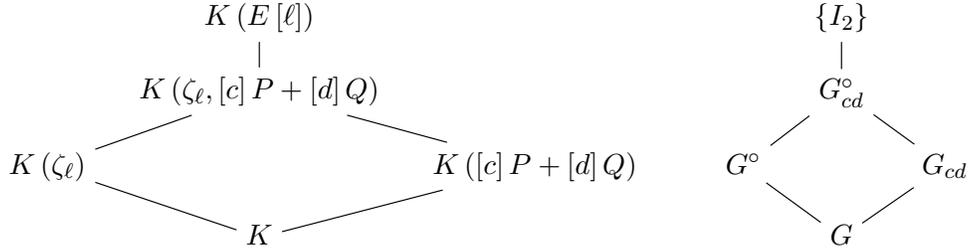
\begin{figure}[h!]\begin{tikzcd}[cramped, sep=small]
                                               & {K\left(E\left[\ell\right]\right)}                                              &                                                                     &  &                               & \left\{ I_{2} \right\}                                         &                            \\
                                               & {K\left(\zeta_{\ell},\left[c\right]P+\left[d\right]Q\right)} \arrow[u, no head] &                                                                     &  &                               & G^{\circ}_{cd} \arrow[u, no head]         &                            \\
K\left(\zeta_{\ell}\right) \arrow[ru, no head] &                                                                                 & {K\left(\left[c\right]P+\left[d\right]Q\right)} \arrow[lu, no head] &  & G^{\circ} \arrow[ru, no head] &                                           & G_{cd} \arrow[lu, no head] \\
                                               & K \arrow[ru, no head] \arrow[lu, no head]                                       &                                                                     &  &                               & G \arrow[ru, no head] \arrow[lu, no head] &                         
	\end{tikzcd}
 \caption{The diagrams of subfields of $K\left(E\left[\ell\right]\right)$ containing $K$ and their corresponding Galois groups in $\Gal\left( K\left(E\left[\ell\right]\right)/ K\right)$}
 \label{fig-1}
 \end{figure}
\end{lemma}


\begin{proof}
	(a), (b), and (c) follow by the definitions of $\rho_{\calB}$ and the Weil pairing directly. In fact, the Weil pairing $e_{\ell}: E\left[ \ell\right] \times E\left[ \ell\right] \to  \{z \in \mathbb{A}^{1} : z^{\ell}-1 = 0 \}$ is a non-degenerate alternative multi-linear form as a $K$-morphism of varieties, so $e_{\ell}\left(P,Q\right)$ is a primitive $\ell$th root of unity~$\zeta_{\ell}$, and $\zeta_{\ell} \in K\left(E\left[\ell\right]\right)$.
	
	For (d), $G_{c\,d}^{\circ}$ consists of matrices with only one eigenvalue $1$ and such non-trivial subgroups of $\GLl$ are conjugates of $\left\langle U \right\rangle$. Hence, $G_{c\,d}^{\circ}$ is trivial, or of order $\ell$.

	(e) is obvious.
\end{proof}


Now, we study the structures of certain subgroups of $\GLl$ which are corresponding to the subgroups of $ G\left(\calB\right)$ of our interest.


We fix a generator $\alpha$ of $\Flx$. Then, $\Fll =\Fl\left(\sqrt{\alpha}\right)$, so we denote the non-trivial Galois action of $\Gal\left(\Fll/\Fl\right)$ by $\overline{a+b\sqrt{\alpha}}:=a-b\sqrt{\alpha}$ for $a,b\in \Fl$.  We call $\beta\in \Fll$  {\it rational (resp. irrational)} if  $\beta\in \Fl$ (resp. $\beta\notin \Fl$). We let $$U:=\begin{psmallmatrix}1&1\\0&1\end{psmallmatrix}, \text{ and } \diag{a,d}:=\begin{psmallmatrix}a&0\\0&d\end{psmallmatrix} \text{ for } a,d\in\Fl^\times.$$
We recall the Borel subgroup $\Bl(\ell)$, the split Cartan subgroup $\Cs(\ell)$, and the non-split Cartan subgroup $\Cns(\ell)$ of $\GLl$. If there is no confusion, we denote them by $\Bl, \Cs, \Cns$ respectively omitting  $\ell$.

The Borel subgroup is defined by $\Bl := \left\{\begin{psmallmatrix}a&b\\0&d\end{psmallmatrix} \in \GLl\right\}$ which is the normalizer of each subgroup of $\Bl$ containing $U$.

The split Cartan subgroup $\Cs \subseteq \GLl$ is defined by the group of all invertible diagonal matrices. We denote an invertible diagonal matrix $\diag{\alpha^{a},\alpha^{b}}$ by $\diagexp{a,b}$. Then, we see that $\Cs\cong  \left(\bbz / \left(\ell-1\right)\bbz\right)^{2}$ via the group isomorphism, $\operatorname{diagexp} : \left(\bbz / \left(\ell-1\right)\bbz\right)^{2} \to \Cs$ defined by $(a,b)\mapsto \diagexp{a,b}$ in the above. 


The non-split Cartan subgroup $\Cns \subseteq \GLl$ is defined by $$
	\left\{\left(\begin{matrix}a&b\alpha\\b&a\end{matrix}\right):\left(0,0\right) \ne \left(a,b\right) \in \Fl^{2}\right\}.
$$
Then, we see that $\Cns\cong  \left(\Fl\left(\sqrt{\alpha}\right)\right)^{\times}$ via a group isomorphism $\left(\Fl\left(\sqrt{\alpha}\right)\right)^{\times}\rightarrow \Cns$ defined by $a+b \sqrt{\alpha} \mapsto \begin{psmallmatrix}a&b\alpha\\b&a\end{psmallmatrix} $.

We note that for each $\left(a,b\right) \in \Fl^{2}-\{(0,0)\}$, the matrix $\begin{psmallmatrix}a&b\alpha\\b&a\end{psmallmatrix} \in \Cns$ has two conjugate eigenvalues $\delta$ and  $\bar{\delta}$ in $\Fll$.

It is clear that
$$\left|\Bl\right|= \ell(\ell-1)^2,\quad  \left|\Cs\right|=(\ell-1)^2, \quad\text{ and }\quad \left|\Cns\right|=\ell^2-1.$$

\begin{lemma}\label{ab_subgp}
	Let $H$ be  an abelian subgroup of $ \GLl$ such that $\ell \nmid \left| H \right|$. Then, there exists $T\in \GLl$ such that $T^{-1}HT\subseteq \Cs$  or $T^{-1}HT\subseteq \Cns$.
\end{lemma}

\begin{proof}
	Since  $H$ is abelian and $\ell \nmid \left| H \right|$, all matrices in $H$ are simultaneously diagonalizable over $\Fll$.
	
	If for each $A\in H$, all eigenvalues of $A$ are rational, then $T^{-1}HT \subseteq \Cs$ for some $T\in \GLl$.
	
	Suppose that there exists $A \in H$ with  two irrational eigenvalues. There exists  $T_{1} \in \operatorname{GL}_{2}\left(\Fll\right)$ such that $T_{1}^{-1}HT_{1}$ consists of diagonal matrices over $\Fll$. So $D := T_{1}^{-1} A T_{1} = \diag{\lambda,\overline{\lambda}}$, for some irrational element $\lambda \in \Fll$, and  for any $A' \in H$, $D':= T_{1}^{-1}A'T_{1}$ is diagonal. Then, the diagonal entries of either $D'$ or $DD'$ are irrational. In other words, the diagonal entries of one of $D'$ or $DD'$ are conjugate to each other. Thus, $D' = \diag{\lambda',\overline{\lambda'}}$ for an eigenvalue $\lambda'$ of $A'$. Hence, $T_{1}^{-1}HT_{1}$ is contained in $\left\{\diag{\delta,\overline{\delta}}:\delta \in \Fllx\right\}$. Since $\Fllx$ is cyclic, there is an irrational element $\mu \in \Fllx$ such that $T_{1}^{-1}HT_{1} = \left\langle \diag{\mu,\overline{\mu}} \right\rangle$. We let $\mu = a+b\sqrt{\alpha}$ for some $a,b\in \Fl$. There are $t,u,v \in \Fl$ such that $\begin{psmallmatrix}t&u\\v&2a-t\end{psmallmatrix} = T_{1}\diag{\mu,\overline{\mu}} T_{1}^{-1} \in H$. We let $P := \begin{psmallmatrix}u&u\\a+b\sqrt{\alpha}-t&a-b\sqrt{\alpha}-t\end{psmallmatrix}$, $Q := \begin{psmallmatrix}b\alpha&b\alpha\\b\sqrt{\alpha}&-b\sqrt{\alpha}\end{psmallmatrix}$, $R := \begin{psmallmatrix}t&u\\v&2a-t\end{psmallmatrix}$, and $S := \begin{psmallmatrix}a&b\alpha\\b&a\end{psmallmatrix}$. Then, we can see that $P^{-1}RP = \diag{\mu,\overline{\mu}} = Q^{-1}SQ$ and that $T := PQ^{-1}$ is Gal$(\Fll/\Fl)$-invariant. Hence, $T\in \GLl$ and $T^{-1}HT = \left\langle S \right\rangle \subseteq \Cns$.
\end{proof}

We denote the normalizers of $\Cs$ and  of $\Cns$  in $\GLl$ by $\Ns$ and $\Nns$, respectively. Then,
$$\Ns= \left\langle \Cs, \begin{psmallmatrix}0&1\\1&0\end{psmallmatrix}\right\rangle \text{ and } \Nns=\left\langle \Cns, \begin{psmallmatrix}1&0\\0&-1\end{psmallmatrix}\right\rangle.$$

\begin{lemma}\label{normalizers}
	Let $H$ be a subgroup of $\Cs$ or $\Cns$ such that $H \not\subseteq \left\{ kI_{2}: k \in \Flx \right\}$ and $\mathscr{N}$ be the normalizer of $H$ in $\GLl$. Then, $$\mathscr{N} \subseteq \begin{cases} \Ns,  &\text{ if }H \subseteq \Cs, \\
 \Nns,  &\text{ if } H \subseteq \Cns.\end{cases}$$
\end{lemma}

\begin{proof}
	Suppose  $H \subseteq \Cs$. Then since $H \not\subseteq \left\{ kI_{2}: k \in \Flx \right\}$,  there exists $B = \diag{a,d} \in H$ where $a,d\in \Fl^\times$ such that  $a\ne d$. For any $X \in \mathscr{N}$, since $XBX^{-1} \in H \subseteq \Cs$,   we conclude that $XBX^{-1}$ is $B$ or $\diag{d,a}$ considering the eigenvalues of $B$. Then, in either case, we can show that $X \in \Ns$ by direct calculation, which implies that $\mathscr{N} \subseteq \Ns$.
	
	Suppose $H \subseteq \Cns$. Then, there exists  $B' = \begin{psmallmatrix}a&b\alpha\\b&a\end{psmallmatrix} \in H$ where $a,b\in \Fl$ such that $b\ne 0$. For any $X \in \mathscr{N}$, since $XB'X^{-1} \in H \subseteq \Cns$, we conclude that $XB'X^{-1}$ is $B'$ or $\begin{psmallmatrix}a&-b\alpha\\-b&a\end{psmallmatrix}$ considering the eigenvalues of $B'$. Then, in either case, we can show  that $X \in \Nns$ by direct calculation, which implies that $\mathscr{N} \subseteq \Nns$.
\end{proof}


Next, we investigate the structures of subgroups of $\GLl$ that  satisfy certain conditions when its intersection with $\SLl$ is considered in the following lemmas.

\begin{lemma}\label{cyclic}
	Let $H$ be a subgroup of $\SLl$ such that $2,\ell \nmid \left| H \right|$.  Then, $H$ is cyclic.
\end{lemma}
\begin{proof}
	If $H = \left\{I_{2}\right\}$, it is trivially true. We assume that $H\ne \left\{I_{2}\right\}$. Then, $H$ is solvable by the Feit-Thompson theorem(\cite{TJ62}). So we let  $n_{0}$ be the smallest positive integer such that $H^{(n_0)} = \left\{I_{2}\right\}$ where $H^{(n)}$ denotes the $n$th derived subgroup of $H$. Since $H$ is not trivial, $n_{0}>0$ and the non-trivial normal subgroup $N := H^{(n_0-1)}\subseteq H$ is abelian. Since $\ell \nmid \left| N \right|$, all matrices in $N$ are simultaneously diagonalizable over $\Fll$. Thus, there exists $P \in \GL{\Fll}$ such that $P^{-1}NP$ consists of diagonal matrices  over $\Fll$. Since any subgroup of $\Fllx$ is cyclic and $\det\left(P^{-1}NP\right) = \left\{ 1 \right\}$, we have that $P^{-1}NP$ is cyclic. Let $D$ be a generator of $P^{-1}NP$. Considering the eigenvalues of $D$ for each $X\in P^{-1}HP$, there exists $\epsilon\left(X\right) \in \left\{\pm 1\right\}$ such that $XDX^{-1} = D^{\epsilon\left(X\right)}$ since $P^{-1}NP \unlhd P^{-1}HP$, so this defines a group homomorphism $\epsilon : P^{-1}HP \to \left\{\pm1\right\}$. Since $\left| H \right|$ is odd, $\epsilon$ is trivial. So any $X \in P^{-1}HP$ commutes with elements of $P^{-1}DP$. Since $N$ is a non-trivial subgroup of $\SLl$ of odd order, $D$ has two distinct diagonal entries (otherwise, $D=\pm I_2$ and $-I_2$ has order $2$). Hence, any $X \in P^{-1}HP$ is diagonal and $H$ is abelian. By Lemma~\ref{ab_subgp}, we conclude that $H$ is contained in $\Cs \cap \SLl$ or $\Cns \cap \SLl$ up to conjugacy and this completes the proof since $\Cs \cap \SLl$ and  $\Cns \cap \SLl$ are cyclic.
\end{proof}

\begin{lemma}\label{NsNns}
	Let  $H $ be a subgroup of $\GLl$ such that  $2,\ell\nmid \left| H ^{\circ} \right|$ where $H^{\circ} = H \cap \SLl$. Then, there exists $T\in \GLl$ such that $T^{-1}HT \subseteq\Ns$ or $T^{-1}HT \subseteq\Nns$.
\end{lemma}

\begin{proof}
	If $H^{\circ} \subseteq \left\{\pm I_{2}\right\}$, then $H$ is abelian since $H/H^{\circ}$ is isomorphic to a subgroup of $ \Flx$ which is cyclic. By Lemma~\ref{ab_subgp}, there exists $T \in \GLl$ such that $T^{-1}HT$ is contained in $\Cs$ or $\Cns$.
	
	If $H^{\circ} \not \subseteq \left\{\pm I_{2}\right\}$, then $H^{\circ} \not \subseteq \left\{ kI_{2}: k\in \Flx\right\}$. So by Lemma~\ref{cyclic},  $H^{\circ}$ is cyclic and by Lemma~\ref{ab_subgp}, there exists $T \in \GLl$ such that $T^{-1} H^{\circ} T$ is contained in $\Cs$ or $\Cns$. Let $\mathscr{N}_{\GLl}\left(T^{-1}H^{\circ}T\right)$ be the normalizer of $T^{-1}H^{\circ}T$ in $\GLl$. Then, since $T^{-1}H^{\circ}T \not \subseteq \left\{ kI_{2}: k\in \Flx\right\}$, Lemma~\ref{normalizers} implies that $ \mathscr{N}_{\GLl}\left(T^{-1}H^{\circ}T\right) \subseteq \Nns \text{ or }\Ns$.  Since $H^{\circ} \unlhd H$, we have that $ T^{-1}HT \subseteq \mathscr{N}_{\GLl}\left(T^{-1}H^{\circ}T\right)$, which completes the proof.
\end{proof}

At last, the following elementary lemma will be useful.

\begin{lemma}\label{SL}
	$\SLl$ is generated by $U=\begin{psmallmatrix}1&1\\0&1\end{psmallmatrix}$ and $U^{t}=\begin{psmallmatrix}1&0\\1&1\end{psmallmatrix}$.
\end{lemma}
\begin{proof}

First, we note that for each $s\in\Fl$,
$$\begin{pmatrix}1&s\\0&1\end{pmatrix}=U^s, \text{ and } \begin{pmatrix}1&0\\s&1\end{pmatrix}=(U^t)^s,$$
and for any non-zero $s\in \Fl$,
	\begin{align*}
		\left(\begin{matrix}0&-s^{-1}\\s&0\end{matrix}\right)
		&= \left(\begin{matrix}1&-s^{-1}\\0&1\end{matrix}\right) \left(\begin{matrix}1&0\\s&1\end{matrix}\right) \left(\begin{matrix}1&-s^{-1}\\0&1\end{matrix}\right) \text{ and }\\
		\left(\begin{matrix}s&0\\0&s^{-1}\end{matrix}\right) &= \left(\begin{matrix}1&0\\-1&1\end{matrix}\right) \left(\begin{matrix}1&1\\0&1\end{matrix}\right) \left(\begin{matrix}1&0\\-1&1\end{matrix}\right) \left(\begin{matrix}0&-s^{-1}\\s&0\end{matrix}\right).
	\end{align*}
	For any $\begin{psmallmatrix}a&b\\c&d\end{psmallmatrix}\in \SLl$, at least one of $a$ and $c$ is non-zero. If $a\ne 0$, then $$
		\left(\begin{matrix}a&b\\c&d\end{matrix}\right)
		= \left(\begin{matrix}1&0\\ca^{-1}&1\end{matrix}\right) \left(\begin{matrix}a&b\\0&d-bca^{-1}\end{matrix}\right)
		= \left(\begin{matrix}1&0\\ca^{-1}&1\end{matrix}\right) \left(\begin{matrix}a&b\\0&a^{-1}\end{matrix}\right)
		= \left(\begin{matrix}1&0\\ca^{-1}&1\end{matrix}\right) \left(\begin{matrix}1&ab\\0&1\end{matrix}\right) \left(\begin{matrix}a&0\\0&a^{-1}\end{matrix}\right).
	$$ If $c\ne 0$, then $$
		\left(\begin{matrix}a&b\\c&d\end{matrix}\right)
		= \left(\begin{matrix}a&\left(ad-1\right)c^{-1}\\c&d\end{matrix}\right)
		= \left(\begin{matrix}1&ac^{-1}\\0&1\end{matrix}\right) \left(\begin{matrix}0&-c^{-1}\\c&d\end{matrix}\right)
		= \left(\begin{matrix}1&ac^{-1}\\0&1\end{matrix}\right) \left(\begin{matrix}1&0\\-cd&1\end{matrix}\right)\left(\begin{matrix}0&-c^{-1}\\c&0\end{matrix}\right).
	$$
\end{proof}

\subsection{The proof of Proposition~\ref{woRCM_main}}

First, we consider the case when $\ell\geq 5$ is a prime and an elliptic curve $E/K$ contains a torsion point $R$ of order $\ell$ such that $2\nmid \left[K\left(R\right):K\right]$.

\

Recall that $ K(E[\ell])$ contains  a primitive $\ell$th root of unity $\zeta_\ell$.


\begin{prop}\label{classify}
	Let $K$ be a number field, $E/K$  be an elliptic curve  over $K$, and $\ell \ge 5$ be a prime. If there exists $R\in E\left[\ell\right]-\left\{O\right\}$ of $E$ such that $2\nmid \left[K\left(R\right):K\right]$.
Then, there is a basis $\calB$ of $E\left[\ell\right]$ such that $G\left(\calB\right)$ is contained in $\Bl$, $\Ns$, or $\Nns$.
\end{prop}
\begin{proof} 
	We let  $K_{0} = K\left(\zeta_{\ell}\right)$. First, we show that the extension $K_{0}\left(R\right)$ is normal over $K_{0}$. Suppose not. First, note that for any $P \in E\left[\ell\right] -\left\langle R\right\rangle$, the set $\left\{ P,R\right\}$ is a basis for $E\left[\ell\right]$.  We fix a basis $\calB':=\left\{ P,R\right\}$ and let $H := G\left(\calB'\right)$. Then since $K_{0}\left(R\right)$ is not normal over $K_{0}$, $H_{01}^{\circ} \subseteq H^{\circ}$ is not a normal subgroup. So $H_{01}^{\circ}$  is non-trivial and there exists $A \in H^{\circ} - \mathscr{N}_{H^{\circ}}\left(H_{01}^{\circ}\right)$, where $\mathscr{N}_{H^{\circ}}\left(H_{01}^{\circ}\right)$ is the normalizer of $H_{01}^{\circ}$ in $H^{\circ}$. Moreover, since $H_{01}^{\circ}$ is a non-trivial subgroup of $\left\langle U \right\rangle$ of order $\ell$ by Lemma~\ref{easy}(d), we see that $H_{01}^{\circ} = \left\langle U \right\rangle$.  Then,  $\rho_{\calB'}^{-1}\left(AUA^{-1}\right)$ fixes $R^{\rho_{\calB'}^{-1}\left(A\right)}$ but not $R$. In fact, if it fixes $R$, then it must be the identity element, which is not true. Hence, we set a new basis $\left\{R^{\rho_{\calB'}^{-1}\left(A\right)},R\right\}$  for $E\left[\ell\right]$ and let  $\Gamma := G\left(\left\{R^{\rho_{\calB'}^{-1}\left(A\right)},R\right\}\right)$. Figure~\ref{fig-2} below shows the diagrams of these extensions over $K$ and their corresponding Galois groups when $K_0(R)$ is not a normal extension of $K_{0}$. Since  $AUA^{-1}$ has order $\ell$, we conclude that $\Gamma$ and so $\Gamma^{\circ}$ contains both $U$ and $U^{t}$.

	{\tiny 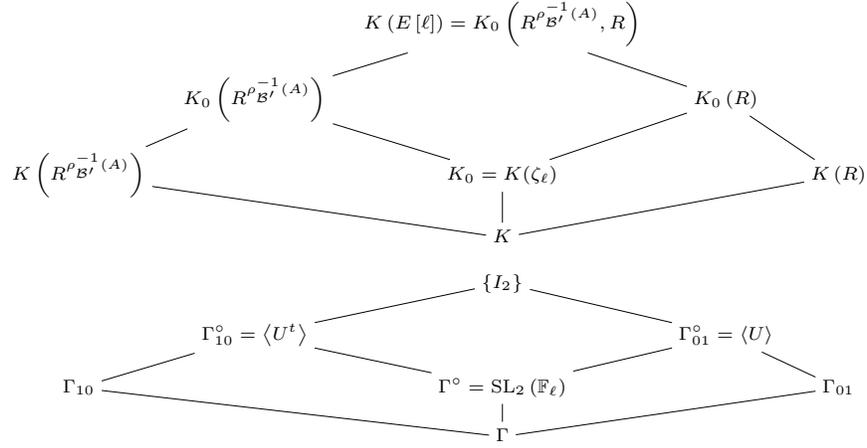
\begin{figure}[h!]\begin{tikzcd}[cramped, sep=small]
                                                                       &                                                                                       & {K\left(E\left[\ell\right]\right) = K_{0}\left(R^{\rho_{\calB'}^{-1}\left(A\right)} ,R \right)} &                                                                        &                                     \\
                                                                       & {K_{0}\left(R^{\rho_{\calB'}^{-1}\left(A\right)}\right)} \arrow[ru, no head] &                                                                                                 & {K_{0}\left(R\right)} \arrow[lu, no head]                     &                                     \\
K\left(R^{\rho_{\calB'}^{-1}\left(A\right)}\right) \arrow[ru, no head] &                                                                                       & K_{0}=K(\zeta_\ell) \arrow[ru, no head] \arrow[lu, no head]                                                   &                                                                        & K\left(R\right) \arrow[lu, no head] \\
                                                                       &                                                                                       & K \arrow[u, no head] \arrow[llu, no head] \arrow[rru, no head]                                  &                                                                        &                                     \\
                                                                       &                                                                                       & \left\{ I_{2} \right\}                                                                                               &                                                                        &                                     \\
                                                                       & \Gamma_{10}^{\circ} =\left\langle U^{t} \right\rangle \arrow[ru, no head]             &                                                                                                 & \Gamma^{\circ}_{01} = \left\langle U \right\rangle \arrow[lu, no head] &                                     \\
\Gamma_{10} \arrow[ru, no head]                                        &                                                                                       & \Gamma^{\circ} = \SLl \arrow[lu, no head] \arrow[ru, no head]                                   &                                                                        & \Gamma_{01} \arrow[lu, no head]     \\
                                                                       &                                                                                       & \Gamma\arrow[rru, no head] \arrow[u, no head] \arrow[llu, no head]                             &                                                                        &                                    
\end{tikzcd}
\caption{The diagrams of subfields of $K\left(E\left[\ell\right]\right)$ containing $K$ and their corresponding Galois groups in $\Gal\left( K\left(E\left[\ell\right]\right)/ K\right)$ when $K_0(R)$ is not normal over~$K_{0}$}
\label{fig-2}
\end{figure}}
 \noindent Thus,  $\Gamma^{\circ} = \SLl$ by Lemma~\ref{SL}, and $2  ~\big|~ \ell^{2}-1 = \left[\Gamma^{\circ}:\Gamma^{\circ}_{01}\right] ~ \big|~ \left[\Gamma:\Gamma_{01}\right] = \left[K\left(R\right):K\right]$, which is a contradiction to our assumption. Therefore, $K_{0}\left(R\right)$ is normal over $K_{0}$.
 
	Now, we consider two cases. First, if $\ell ~\big|~ \left[K\left(E\left[\ell\right]\right):K\right]$, then there exists basis $\calB$ for $E\left[\ell\right]$ such that $U \in G := G\left(\calB\right)$. Referring to Figure~\ref{fig-2}, since $\left\langle U \right\rangle = G_{01}^{\circ} \unlhd G^{\circ}$, we have that $U \in G^{\circ} \subseteq \mathscr{N}_{\GLl}\left(\left\langle U \right\rangle\right) = \Bl$. Since $ G^{\circ} \unlhd G$, we have that $G \subseteq \mathscr{N}_{\GLl}\left(G^{\circ}\right)$ and  $\mathscr{N}_{\GLl}\left(G^{\circ}\right)= \Bl$, since $U \in G^{\circ}$.
	
	Second, if $\ell \nmid \left[K\left(E\left[\ell\right]\right):K\right]$, we fix a basis  $\calB' :=\left\{P,R\right\}$ for $E\left[\ell\right]$ for some $P \in E\left[\ell\right] - \left\langle R \right\rangle$ and  let $H  := G\left(\calB'\right)$. Then, $\ell \nmid \left| H^\circ \right|$ and  $H^{\circ}_{01} =\left\{I_{2}\right\}$  by Lemma~\ref{easy}(d) since $\left| H ^{\circ}_{01} \right| \mid \ell$ but $\ell\nmid |H|$. Hence,  $H^\circ\cap H_{01}=\left\{I_{2}\right\}$ and the semi-direct product $H^{\circ} \rtimes H_{01}$ is a subgroup of $H$ and $\left| H ^{\circ} \right|$ divides $\left[H:H_{01}\right]$. Since $\left[H:H_{01}\right]$ is odd, so is $\left| H ^{\circ} \right|$. So by Lemma~\ref{NsNns}, there exists $T \in \GLl$ such that $T^{-1} H T$ is contained in $\Ns$ or $\Nns$. Therefore, there is a basis $\calB$ for $E\left[\ell\right]$ such that $G\left(\calB\right)$ is contained in $\Ns$ or $\Nns$.
\end{proof}

%

Next, we observe a sufficient condition for  the degree $\left[K\left(R\right):K\right]$ to be divisible by $2$ or $3$ for a point $R \in E\left[\ell\right] - \left\{ O \right\}$ where  $\ell\geq 11$  is an unramified prime in $K$,  by applying Genao's result \cite[Theorem~2]{G22} which describes the image of the inertia group based on the deduction types.

\begin{prop}\label{not_Bl}
	Let $K$ be a number field and $E/K$ be an elliptic curve over $K$. Suppose $\ell\ge 11$ is a prime satisfying the following two conditions: \begin{enumerate}[\normalfont (i)]
		\item $\bbq\left(\zeta_{\ell}\right) \cap K = \bbq$.
\item $\ell$ is unramified in $K$.
	\end{enumerate}
	If  $ G\left(\calB\right)$ is contained in $\Ns$, or in $\Nns$ but not $\Cs$ for some basis $\calB$ of $E\left[\ell\right]$, then for any point $S \in E\left[\ell\right] - \left\{ O \right\}$, the degree $\left[K\left(S\right):K\right]$ is divisible by $2$ or $3$.
\end{prop}
\begin{proof}	
	First, suppose $G:= G\left(\calB\right) \subseteq \Nns$. If $A \in \Nns$ has an eigenvalue $1$, then the other eigenvalue of $A$ is $1$ or $-1$. So for any non-zero row vector $\left(c\,d\right) \in \Fl^{2}$, $G_{c\,d}$ is contained in $\left\{\begin{psmallmatrix} \pm1&0\\0&1\end{psmallmatrix} \right\}$ or $\left\{\begin{psmallmatrix} 1&0\\0&\pm1\end{psmallmatrix} \right\}$. Moreover, $\Cns \cap G_{c\,d} = \left\{ I_{2}\right\}$. By Lemma~\ref{NsNns}, we know that $\Cns^{e}:=\{A^e : A\in \Cns\}$  is a normal subgroup of $G$ since $G \subseteq \Nns$, and by \cite[Theorem~2]{G22} for some $e\in \left\{ 1,2,3,4,6 \right\}$, $$
		2 ~\big|~ (\ell^{2}-1)/e = \left| \Cns^{e} \right| = \left[ \Cns^{e} : \Cns^{e} \cap G_{cd}\right] = \left[ \Cns^{e}G_{cd} : G_{cd}\right] ~\big|~ \left[ G : G_{cd} \right].
	$$ This implies that $\left[K\left(S\right):K\right]$ is divisible by $2$, since $\left[K\left(S\right):K\right]=\left[ G : G_{cd} \right]$ for some nonzero vector $(c\,d)\in \Fl^{2}$.

	Now, suppose $G$ is contained in $\Ns$ but not $\Cs$. The index of the subgroup $\Delta := G \cap \Cs$ in $G$ is $2$. By direct computation, we can show that $G_{01}$ and $G_{10}$ are contained in $\Delta$. So $\left[G:G_{10}\right] = \left[G:\Delta\right] \left[\Delta:G_{10}\right]$, which is even. For any non-zero $s \in \Fl$, direct computation shows that $\left| G_{1s} \right| \mid 2= \left[ G:\Delta\right]$ and $|\Delta|\mid [G:G_{1 s}]$. Thus, it is enough to show that $\left| \Delta \right|$ is divisible by $2$ or $3$.  The assumption (i) implies that there exists $A \in G$ such that $\det A =\alpha$, where $\alpha$ is a generator of $\Fl^\times$.

If $\ell \equiv 1\pmod{4}$, then $2 \mid (\ell-1)/2$ and  $| A^{2}| ~\big|~ \left| \Delta \right|$, so $2 \mid \left| \Delta \right|$.

If $\ell \equiv 3\pmod{4}$,  we first claim that  $\ell = 19$ or  $\rho\left(I_{K_{\mathfrak{p}}}\right) \subseteq \Cs$ where $K_{\mathfrak{p}}$ is the completion of $K$ at a prime  $\mathfrak{p}$ above $\ell$, $I_{K_{\mathfrak{p}}}$ is the inertia group of $\mathfrak{p}$, and $\rho = \rho_{\calB}$, following the proof of \cite[Theorem~2]{G22}.
Suppose not, i.e., suppose that $\ell \ne 19$ and  $\rho\left(I_{K_{\mathfrak{p}}}\right)$ is contained in $\Ns$ but not $\Cs$. Since $I_{K_{\mathfrak{p}}}$ is cyclic, there exists an element $B \in \Ns -\Cs$ which generates $\rho\left(I_{K_{\mathfrak{p}}}\right)$. By \cite[Theorem~2]{G22}, $\rho\left(I_{K_{\mathfrak{p}}}\right)$ contains $T\diagexp{0,e}T^{-1}$ for some $e \in \left\{1,2,3,4,6\right\}$ and for some $T\in \GLl$. If $T\diagexp{0,e}T^{-1} \in G \subseteq \Ns$, then $T\diagexp{0,e}T^{-1}$ is $\diagexp{0,e}$ or $\diagexp{e,0}$. If $T\diagexp{0,e}T^{-1} = \diagexp{0,e}$, we have that $\diagexp{e,0} = B\diagexp{0,e}B^{-1} \in \rho\left( I_{K_{\frak{p}}}\right)$ and $\Cs^{e} = \left\{ A^{e}: A\in \Cs \right\} \subseteq \rho\left( I_{K_{\frak{p}}}\right)$. If $T\diagexp{0,e}T^{-1} = \diagexp{e,0}$, by the same argument, we can show that $\Cs^{e} \subseteq \rho\left( I_{K_{\frak{p}}}\right)$. Hence,  since  $B^{2} \in \Cs$ and $B^{2\left(\ell-1\right)} = I_{2}$, we have that $$
		\left(\frac{\ell-1}{\gcd\left(\ell-1,e\right)}\right)^{2} = \left| \Cs^{e} \right| ~\big|~ \left| \rho\left(I_{K_{\mathfrak{p}}}\right) \right| = \left| B\right| ~\big|~ 2\left(\ell-1\right),
	$$
which implies that $\ell-1$ divides $2e^{2}$. Thus, $\ell-1$ divides $32$ or $72$. Since  $\ell-1 \equiv 2 \pmod{4}$, we have that $\ell \in \left\{ 3, 7, 19 \right\}$. But this contradicts that $\ell \ge 11$ and $\ell\ne 19$. Thus, $\ell = 19$ or $\rho\left(I_{K_{\mathfrak{p}}}\right) \subseteq \Cs$. If $\ell=19$, then $3 \mid \frac{\ell-1}{\gcd\left(\ell-1,e\right)}= \left| \diagexp{0,e} \right| ~\big|~ \left| \Delta \right|$. If $\rho\left(I_{K_{\mathfrak{p}}}\right) \subseteq \Cs$, then we have that $$
		2 \mid \ell -1 = \left| \Flx \right| = \left| \left(\det \circ \rho\right)\left(I_{K_{\frak{p}}}\right) \right| ~\big|~ \left| \Delta \right|,
	$$ since the assumption (i) implies that $\left(\det \circ \rho\right)\left(I_{K_{\frak{p}}}\right) = \Flx$. This completes the proof.
\end{proof}

\begin{proof}[Proof of Proposition~\ref{pre_main}]
	There is a constant $N'_{K}$ depending only on $K$ such that any prime $\ell > N'_{K}$ satisfies the conditions (i) and (ii) of Proposition~\ref{not_Bl}. Then, the proof follows from  Proposition~\ref{classify} and Proposition~\ref{not_Bl}.
\end{proof}

Now, we characterize when $G\left(\calB\right) \subseteq \Bl$ up to conjugacy for a sufficiently large prime $\ell$ if $K$ has no RCM. We recall the following result which gives a lower bound of such a prime $\ell$ depending on $K$.

\begin{thm}[{\cite[Theorem~1]{LV14}}]\label{LV14}
	Let $K$ be a number field. Then, there exists a finite set $S_{K}$ of primes depending only on $K$ such that for a prime $\ell \not \in S_{K}$, and an elliptic curve $E/K$ for which $E\left[\ell\right] \otimes \overline{\Fl}$ is reducible with associated character $\psi$ of  degree $1$, one of the following holds:\begin{enumerate}[\arabic*.]
		\item There exists an elliptic curve $\mathscr{E}/K$  whose CM field is contained in $ K$, with an $\ell$-adic degree $1$ associated character whose mod-$\ell$ reduction $\phi$ satisfies: $$
			\psi^{12} = \phi^{12}
		$$
		\item GRH fails for $K\left[\sqrt{-\ell}\right]$ and $\psi^{12} = \cyc^{6}$, where $\cyc$ denotes an $\ell$-cyclotomic character of $\Gal\left(\overline{\bbq}/\bbq\right)$.

	\end{enumerate}
(Refer to \cite[\S1]{LV14} for the associated characters of degree $d$.)
\end{thm}


\begin{prop}\label{Bl}
	Let $K$ be a number field without RCM, $S_{K}$ a set of primes given in Theorem~\ref{LV14}, and $\ell\ge 11$ a prime satisfying the following three conditions;
\begin{enumerate}[\normalfont (i)]
	\item $\ell \not \in S_{K}$,
	\item $\ell$ is unramified in $K$, and
	\item $\bbq\left(\zeta_{\ell}\right) \cap K = \bbq$.
\end{enumerate}
	Then, for any elliptic curve $E$ over $K$ with  a point $R\in E\left[\ell\right]-\left\{O\right\}$ such that $2, 3\nmid \left[K\left(R\right):K\right]$, the followings hold:
\begin{enumerate}[\normalfont (a)]
    \item There is a basis $\calB$ of $E\left[ \ell \right]$ such that $G := G\left(\calB\right)$ is $\left\langle \Delta, U \right\rangle$ where $\Delta := G \cap \Cs$,
	\item  $\Delta$ is  either $\Delta_{1} := \diagexp{\left\langle \left(2\tau,2\tau\right), \left(0,\frac{\ell-1}{2\tau}\right) \right\rangle}$ or $\Delta_{2} := \mathrm{diagexp}\left(\left\langle \left(2\tau,2\tau\right), \right.\right.$ $ \left.\left.\left(\frac{\ell-1}{2\tau},0\right) \right\rangle\right),$ where the constant $\tau$ is defined by $$
		\tau = \begin{cases}1 & \text{ if } \ell \equiv 2 \pmod{3}, \\ 3 & \text{ if } \ell \equiv 1 \pmod{3},\end{cases}
	$$
	\item $\frac{\ell-1}{2\tau}$ divides $\left[ K\left(S\right) : K\right]$ for any $S \in E\left[\ell\right] -\left\{O\right\}$, and
	\item $\ell\equiv 3\pmod{4}$, $\ell \not \equiv 1\pmod{9}$, and $\ell \in \PDI_{2}\left(K\right)$.
\end{enumerate}
\end{prop}

\begin{proof}
First, we show that there exists a basis $\calB$ of $E\left[ \ell \right]$ such that $G := G\left(\calB\right)$ is $\Delta$ or $\left\langle \Delta, U \right\rangle$. By Proposition~\ref{pre_main}, there exists a basis  $\calB'$ for $E\left[\ell\right]$ such that $H := G\left(\calB'\right) \subseteq \Bl$. 

If $\ell\nmid \left| H \right|$, then we note that $A\in H$ has only one eigenvalue $1$ if and only if $A=I_2$. Also, since we can show that for any $A,B \in H \subseteq \Bl$, $ABA^{-1}B^{-1}$ has only one eigenvalue~$1$ by direct calculation, $H$  is abelian. Since $\ell\nmid \left| H \right|$ and all matrices in $H$ have rational eigenvalues, all elements of $H$ are simultaneously diagonalizable over $\Fl$. So $T^{-1} H T \subseteq \Cs$ for some $T\in \GLl$, i.e., $G := G\left(\calB\right) \subseteq \Cs$ for some basis $\calB$ for $E\left[ \ell \right]$.

If $\ell~\big|~\left| H \right|$, then $H$ contains a subgroup of order $\ell$ by Cauchy's theorem. Since $\Bl$ has the unique subgroup $\left\langle U\right\rangle$ of order $\ell$,  we have that $U \in H$. Therefore, $\begin{pmatrix}a&0\\0&d\end{pmatrix} = \begin{pmatrix}a&b\\0&d\end{pmatrix} U^{-ba^{-1}}\in H$ for any $\begin{pmatrix}a&b\\0&d\end{pmatrix} \in H$. Hence, for  $G := H$,  we have that  $G= \left\langle G \cap \Cs, U \right\rangle$.

Therefore, we have shown  that for some basis $\calB = \left\{ P,Q \right\}$ for $E\left[ \ell \right]$,
\begin{align}\label{gp-G}
    \text{ the group } G:=G(\calB) \text{ is } \Delta \text{ or } \left\langle \Delta, U \right\rangle.
\end{align}

Next, we prove (b)-(d). To prove (b), first,  we recall the situation so that we can apply Theorem~\ref{LV14} under the assumption (i): The representation $E\left[\ell\right] \otimes \overline{\Fl}$ of $\calG := \Gal\left(K\left(E\left[\ell\right]\right)/K\right)$ has exactly two associated characters $\psi_{1}$ and $\psi_{2}$ of degree $1$. They satisfy that $\rho\left(\sigma\right) = \begin{psmallmatrix} \psi_{1}\left(\sigma\right) & * \\ 0 & \psi_{2}\left(\sigma\right) \end{psmallmatrix}$ 
for all $\sigma \in \calG$ where $\rho = \rho_{\calB}$. Since $K$ has no RCM,  the second case of Theorem~\ref{LV14} holds, which implies that for each  $i=1,2$, $\psi_{i}^{12}\left(\sigma\right) = \cyc^{6}\left(\sigma\right) = \det\left(\rho\left(\sigma\right)\right)^{6}$ for all $\sigma \in \calG$, and equivalently,
\begin{align}\label{cong}
\text{ for all }\diagexp{u,t} \in \Delta,\quad  12 u \equiv 12t \equiv 6\left(u+t\right) \pmod{\ell-1}.
\end{align}
	
	Since $\Flx$ is a cyclic group generated by $\alpha$, $\Delta^{\circ} := \Delta \cap G^{\circ}$ is generated by a single matrix, say $\diagexp{\frac{\ell-1}{m},\frac{1-\ell}{m}}$ for some positive divisor $m$ of $\ell-1$. Next, we claim that there exists $A\in \Delta$ such that $\det\left(A\right)=\alpha$. In fact, since $\bbq\left(\zeta_{\ell}\right) \cap K = \bbq$ by (ii), $\Gal\left( K\left(\zeta_{\ell}\right)/K\right) \cong \Gal\left( \bbq\left(\zeta_{\ell}\right)/\bbq\right)$ and $\cyc\left(\Gal\left( K\left(\zeta_{\ell}\right)\right)/K\right) = \cyc\left(\Gal\left( \bbq\left(\zeta_{\ell}\right)/\bbq\right)\right) =\Flx = \left\langle \alpha\right\rangle$. Hence, there exists $A'=\begin{pmatrix} a&b\\0&d\end{pmatrix}\in G$ such that $\det A' = \alpha$. If $U \in G$, then $A:=A' U^{-ba^{-1}} \in \Delta$ and $\det\left(A\right) = \alpha$. If $U\notin G$, then $A:=A'\in G = \Delta$ by \eqref{gp-G}. So we let $A=\diagexp{1-t,t}\in \Delta$ for some $t\in \bbz/\left(\ell-1\right)\bbz$. Then, since $\diagexp{1-t,t}^{-a-b}\diagexp{a,b} \in \Delta^{\circ} = \diagexp{\left\langle \left(\frac{\ell-1}{m},\frac{1-\ell}{m}\right) \right\rangle}$ for any $\diagexp{a,b} \in \Delta$, we see that $$\Delta = \diagexp{\left\langle \left(\frac{\ell-1}{m},\frac{1-\ell}{m}\right), \left(1-t,t\right) \right\rangle}.$$  Next, we show that $m=1$ and so $\Delta = \diagexp{\left\langle \left(1-t,t\right) \right\rangle}$.
	
	By \eqref{cong} for $\diagexp{\frac{\ell-1}{m},\frac{1-\ell}{m}}$,  $12\frac{\ell-1}{m} \equiv 0 \pmod{\ell-1}$, and equivalently, $m\mid 12$.

	Recalling \eqref{gp-G}, if $G=\Delta$, then we have that $G_{01} = \left\{ \begin{psmallmatrix}a&0\\0&1\end{psmallmatrix}\in G\right\}$, $G_{10} = \left\{ \begin{psmallmatrix}1&0\\0&d\end{psmallmatrix}\in G\right\}$, and $G_{1k} = \left\{  I_{2} \right\}$ for all non-zero $k \in \Fl$. Therefore, we can get $$
		\left[G:G_{01}\right] = \left| \left\{ d\in \Flx: \begin{psmallmatrix}a&0\\0&d\end{psmallmatrix} \in \Delta \text{ for some }a\in \Flx\right\} \right|,
	$$ $$
		\left[G:G_{1k}\right] = \left| G \right|, \text{ and } \left[G:G_{10}\right] = \left| \left\{ a\in \Flx: \begin{psmallmatrix}a&0\\0&d\end{psmallmatrix} \in \Delta \text{ for some }d\in \Flx\right\} \right|.
	$$ By  (ii), we have that $\ell -1 ~\big|~ \left| G \right| = \left[G:G_{1k}\right]$ for each non-zero $k \in \Fl$, so $2\mid \left[G:G_{1k}\right]$. Moreover, since $\diagexp{\frac{\ell-1}{m},\frac{1-\ell}{m}}\in\Delta=G$, $m$ divides $\left[G:G_{01}\right]$ and $\left[G:G_{10}\right]$. Since the degree $\left[K\left(R\right):K\right]$ is equal to one of $\left[G:G_{01}\right]$, $\left[G:G_{01}\right]$, or $\left[G:G_{1k}\right]$ for some non-zero $k \in \Fl$ by Lemma~\ref{easy} and $2\nmid \left[K\left(R\right):K\right]$ by (iii), we conclude that $2 \nmid m$.
If $\ell \equiv 1\pmod{3}$, then since a similar argument shows that $3\nmid m$ since $3\nmid \left[K\left(R\right):K\right]$ by (iii). If $\ell \equiv 2\pmod{3}$, then clearly $3\nmid m$ since $m$ is a divisor of $\ell-1$.
Therefore, since $m\mid 12$ but $2,3\nmid m$, we conclude that $m=1$ and $\Delta = \diagexp{\left\langle \left(1-t,t\right) \right\rangle}$.

	
If $ G=\left\langle \Delta, U \right\rangle$, then $G_{01} = \left\{ \begin{psmallmatrix}a&b\\0&1\end{psmallmatrix} \in G \right\}$ and $G_{1k'} = \left\{ \begin{psmallmatrix}1&k'\left(1-d\right)\\0&d\end{psmallmatrix} \in G \right\}$ for all $k' \in \Fl$. Hence, $$
		\left[G:G_{01}\right] = \left| \left\{ d\in \Flx: \begin{psmallmatrix}a&0\\0&d\end{psmallmatrix} \in \Delta \text{ for some }a\in \Flx\right\} \right| \text{ and }
	$$$$
		\left[G:G_{1k'}\right] = \ell\cdot \left| \left\{ a\in \Flx: \begin{psmallmatrix}a&0\\0&d\end{psmallmatrix} \in \Delta \text{ for some }d\in \Flx\right\} \right|.
	$$ Since $\ell \ge 5$,  we can show that $\Delta = \diagexp{\left\langle \left(1-t,t\right) \right\rangle}$ by the same argument.

	Note that one of $t$ or $1-t$ is even modulo $\ell-1$.
  Suppose that $t$ is even. Recalling the indices $\left[G:G_{1k'}\right]$ and $\left[G:G_{01}\right]$ described in the above, we have that  for all $k' \in \Fl$, $\left[G:G_{1k'}\right]$ is divisible by  $|1-t|=\frac{\ell-1}{\gcd(1-t, \ell-1)}$ which is even since $1-t$ is odd, and $ \left[G:G_{01}\right]=\left|t\right|$. Since the degree $\left[K\left(R\right):K\right]$ is equal to $\left[G:G_{01}\right]$ or $\left[G:G_{1k'}\right]$ for some  $k' \in \Fl$   by Lemma~\ref{easy} again and $2\nmid \left[K\left(R\right):K\right]$ by (iii), we conclude that  $\left[K\left(R\right):K\right]=\left[G:G_{01}\right]=\left|t\right|$, which should not be divisible by $2$ or $3$.

	By \eqref{cong} for $\diagexp{1-t,t} \in G$, we have  $12t \equiv 6\pmod{\ell-1}$, so $6t \equiv 3\pmod{\frac{\ell-1}{2}}$. If $\ell \equiv 1 \pmod{3}$, then $2t \equiv 1\pmod{\frac{\ell-1}{6}}$ so $t \equiv \frac{\ell+5}{12} \pmod{\frac{\ell -1}{6}}$. Therefore, $$
		\gcd\left(t,\ell-1\right) \mid \gcd\left(\ell+5,\ell-1\right) = \gcd\left(6,\ell-1\right) = 6.
	$$ Thus, since $\left| t\right| = \frac{\ell-1}{\gcd\left(t,\ell-1\right)}$, $|t|$ is divisible by $\frac{\ell-1}{6}$ and  $\left| t\right|$ divides $\ell-1$. Moreover, since $2,3 \nmid \left| t \right|$, we have that $\left| t \right| = \frac{\ell-1}{6}$ and $\gcd\left(6,\frac{\ell-1}{6}\right) = 1$. Hence, $\gcd\left(t,\ell-1\right)=6$ and there is an $u\in \left(\bbz/\left(\ell-1\right)\bbz\right)^{\times}$ satisfying $tu\equiv 6 \pmod{\ell-1}$. Moreover, $\Delta = \diagexp{\left\langle \left(1-t,t\right) \right\rangle} = \diagexp{\left\langle \left(u-6,6\right) \right\rangle}$. By \eqref{cong} for $\diagexp{u-6,6} \in G$, we have that $u\equiv 12\pmod{\frac{\ell-1}{6}}$. Since $\gcd\left(u,\ell-1\right) = 1$, we have that $2,3\nmid u$ and that $u \equiv 12 \pm \frac{\ell-1}{6}\pmod{\ell-1}$. Hence $\left(u-6,6\right) = \left(6 \pm \frac{\ell-1}{6},6\right)$ in $\left(\bbz/\left(\ell-1\right)\bbz\right)^2$. Since $\left\langle \left(6 + \frac{\ell-1}{6},6\right) \right\rangle \subseteq \left\langle \left(6,6\right), \left(\frac{\ell-1}{6},0\right) \right\rangle$ and both groups have the same order $\ell-1 = \operatorname{lcm}\left(\frac{\ell-1}{6}, 6\right) $, we have that $$
		\left\langle \left(6 - \frac{\ell-1}{6},6 \right)\right\rangle
		= \left\langle \left(6,6\right), \left(\frac{\ell-1}{6},0\right) \right\rangle
		= \left\langle \left(6 + \frac{\ell-1}{6},6\right) \right\rangle.
	$$ Finally, $\Delta = \diagexp{\left\langle \left(6,6\right),\left(\frac{\ell-1}{6},0\right) \right\rangle}$.
If $\ell\equiv 2 \pmod{3}$, we can show by a similar argument that $\Delta = \diagexp{\left\langle \left(2,2\right),\left(\frac{\ell-1}{2},0\right) \right\rangle}$. Therefore, in either case, $\Delta= \Delta_{2}$ if $t$ is even.  If $1-t$ is even, we can show that $\Delta= \Delta_{1}$ similarly.

For (c),	
we note that $ \diagexp{2\tau, 2\tau}\in \Delta$ by (b), so  $|2\tau|=\frac{\ell-1}{2\tau}$ divides $\left[G:G_{01}\right]$ and $\left[ G: G_{1k'}\right]$ for any $k' \in \Fl$ recalling those indices. Thus,  $\frac{\ell-1}{2\tau}$ divides $\left[ K\left(S\right): K\right]$ for any $S  \in E\left[\ell\right] - \left\{ O \right\}$.

For (d), by (c) and the condition (iii),  $\frac{\ell-1}{2\tau}$ should not be divisible by $2$ or $3$, which implies that $\ell \equiv 3\pmod{4}$ and $\ell \not \equiv 1 \pmod{9}$. Moreover, the isogeny character of $\left(E, \left\langle Q\right\rangle \right) \in Y_{0}\left(\ell\right)$ referring to Theorem~\ref{M95A} is $\psi_{2}$, and by (b), $\psi_{2}$ satisfies $\psi_{2}^{6} = \cyc^{12}$ for  $\calB = \left\{ P, Q\right\}$  in \eqref{gp-G}. Hence $\ell \in \PDI_{2}\left(K\right)$.
	
To complete the proof of  (a), by \eqref{gp-G}, it is enough to show that $U\in G$. Suppose $U \not \in G$. Then by \eqref{gp-G}, $G = \Delta$ and $\ell \nmid \left| G \right|$. Hence, \cite[Theorem~2]{G22} implies that there exist $e \in \left\{ 1,2,3,4,6\right\}$ and $T \in \GLl$ such that either $\diagexp{0,e}\in T G T^{-1}$ or $\Cns^{e}\subseteq T G T^{-1}$. If $\diagexp{0,e} \in T G T^{-1}$, then $e \in \left\langle \frac{\ell-1}{2\tau} \right\rangle$. So $\ell-1 \mid 2\tau e$. Therefore, $\ell-1$ divides $24$ or $36$. Since $\ell-1 \equiv 2\pmod{4}$ and $\ell \not \equiv 1\pmod{9}$ by (d), we have that $\ell = 3$ or $ 7$, which  contradicts that $\ell \ge 11$. If $T G T^{-1}$ contains $\Cns^{e}$, then $$
		\frac{\ell^{2}-1}{e} = | \Cns^{e} | ~\big|~ \left| G \right| = \ell-1,
	$$ so $\ell+1 \mid e$, which contradicts that $\ell \ge 11$ again. This completes the proof.
\end{proof}

\begin{proof}[Proof of Proposition~\ref{woRCM_main}]
	This follows from Proposition~\ref{Bl}, by letting $N_{K}= \max \left(S'_{K}\right)$ where  $S'_{K}:=S_{K}\cup \{\text{a prime } \ell : \ell \text{ is ramified in } K, \text{ or } \bbq(\zeta_{\ell})\cap K\neq \bbq\}$ which is a finite set.
\end{proof}
\begin{remark}
	In Proposition~\ref{Bl}(b), we note that $\Delta_{1}$ and $\Delta_{2}$  do  appear mutually inclusively. More precisely, for an elliptic curve $E/K$ and a basis $\calB =\left\{P,Q\right\}$ of $E\left[ \ell\right]$, if $G := G\left(\calB\right) = \left\langle \Delta, U \right\rangle$,  then we can show there exist an elliptic curve $E'/K$ and a basis $\calB'$ of $E'\left[ \ell\right]$ such that $G' := G\left(\calB'\right) = \left\langle \Delta_{\operatorname{flip}}, U \right\rangle$ where $\Delta_{\operatorname{flip}} := \left\{\diag{d,a} : \diag{a,d} \in \Delta \right\}$ as follows:
 
	Note that $\Delta_{1}$ and $\Delta_{2}$ are the flips of each other. Since the subspace $\left\langle Q \right\rangle \subseteq E\left[\ell\right]$ is $\Gal\left(K\left(E\left[\ell\right]\right)/K\right)$-invariant, there is a $K$-rational isogeny $\alpha: E \to E'$ with kernel $\left\langle Q \right\rangle$. We denote the dual isogeny of $\alpha$ by $\widehat{\alpha}$. Then, $\widehat{\alpha} \circ \alpha = \left[ \ell\right]$ and $\ker \alpha = \left\langle Q \right\rangle$, so the point $Q' := \alpha\left(P\right) \in E'\left[ \ell\right]$ is non-zero and it is in $\ker \widehat{\alpha}$. Since the order of $\widehat{\alpha}$ is $\ell = \# \ker \alpha$, the kernel $\widehat{\alpha}$ is generated by a $Q'$. Since $\# \ker\widehat{\alpha}$ is of order $\ell$, $Q'$ generates $\ker\widehat{\alpha}$. Similarly, for any $P' \in E'\left[\ell\right] - \left\langle Q'\right\rangle$, $\widehat{\alpha}\left(P'\right)$ generates $\ker \alpha$. Replacing $P'$ by a multiple of $P'$ of an appropriate scalar, we may assume that $\widehat{\alpha}\left(P'\right) = Q$. Then, the set $\calB' = \left\{P',Q'\right\}$ is a basis of $E'\left[\ell\right]$. So far, we have shown that the matrix representation of the linear transformations $\alpha : E\left[\ell\right] \to E'\left[\ell\right]$ and $\widehat{\alpha} : E'\left[\ell\right] \to E\left[\ell\right]$ with respect to the bases $\calB$ and $\calB'$, respectively are both equal to  $\begin{psmallmatrix}0&1\\0&0\end{psmallmatrix}$. We may consider $\sigma \in \Gal\left(\overline{K}/K\right)$ as automorphisms on the $\Fl$-vector spaces $E\left[\ell\right]$ and $E'\left[\ell\right]$. We denote  by $\begin{psmallmatrix}a&b\\c&d\end{psmallmatrix}$ and $\begin{psmallmatrix}a'&b'\\c'&d'\end{psmallmatrix}$ the matrix representations of automorphisms $\sigma$ on $E \left[\ell\right]$ and $E' \left[\ell\right]$ with respect to the bases $\calB$ and $\calB'$, respectively. Since $\alpha$ and $\sigma$ commute, we have the following commutative diagram,

	\begin{center}\begin{tikzcd}
		{E\left[\ell\right]} \arrow[dd, "\sigma"'] \arrow[rr, "\alpha"] &                  & {E'\left[\ell\right]} \arrow[dd, "\sigma"] \\
		& \circlearrowleft &                                            \\
		{E\left[\ell\right]} \arrow[rr, "\alpha"]                       &                  & {E'\left[\ell\right]}                     
	\end{tikzcd}
\end{center}
	and	we have that $$
		\begin{pmatrix}0&a\\0&c\end{pmatrix}
		= \begin{pmatrix}a&b\\c&d\end{pmatrix} \begin{pmatrix}0&1\\0&0\end{pmatrix}
		= \begin{pmatrix}0&1\\0&0\end{pmatrix} \begin{pmatrix}a'&b'\\c'&d'\end{pmatrix}
		= \begin{pmatrix}c'&d'\\0&0\end{pmatrix}.
	$$ Thus, $a = d'$ and $c = c' =0$. Replacing $\alpha$ by $\widehat{\alpha}$, the same argument shows that $a' = d$. For $G':= G\left(\calB'\right)$, we see that $G'$ is contained in $\left\langle \Delta_{\operatorname{flip}}, U \right\rangle$. Moreover, as in the end of  the proof of Proposition~\ref{Bl} we can show that $U \in G'$. Therefore, we conclude that $G' = \left\langle \Delta_{\operatorname{flip}}, U \right\rangle$.
\end{remark}

\section{The proofs  of the main theorems}\label{main_proof}

In this section, we give the proofs of our main theorems, Theorem~\ref{quad_main} and Theorem~\ref{main}. We will give a prime $p_K$ depending on $K$, more specifically, depending on $N_K$ given in Proposition~\ref{woRCM_main}, $\calR\left(K\right)$ given in \eqref{setR}, and Merel's bound on the torsion points over $K$ (\cite{Merel}), and we prove Theorem~\ref{quad_main} by dividing it into two cases; for primes $\ell \leq p_K$ and for primes $\ell >p_K$. For the latter case, we prove that under the assumption of the degree of the extension $L$ over $K$, the $\ell$-torsion subgroups over both $K$ and $L$ are the same as the trivial group by applying Proposition~\ref{woRCM_main}, and for the former case, we prove that  $\ell^{\infty}$-torsion parts of the elliptic curves, upon base change, do not grow by applying Proposition~\ref{small_prime} after proving it below.

We recall Merel's theorem (\cite{Merel}) which gives an uniform upper bound of orders of torsion points  over $K$ of an elliptic curve $E/K$, and \cite[Lemma~3.2]{IK22} which gives an equivalent condition for the $N$-torsion subgroup of $E\left(K\right)$ to contain a certain type of subgroups.

\begin{thm}[{\cite{Merel}}]\label{Merel}
	Let $K$ be a number field. There is a positive integral constant $M_{K}$ satisfying ${E\left(K\right)}_{\tors} = E\left(K\right)\left[M_{K}\right]$ for all elliptic curves $E/K$.
\end{thm}

\begin{defn}[the Merel constant]\label{Merelconstant}
	For a number field $K$, we let $M\left(K\right)$ be the smallest positive  constant  among $M_K$ given in Theorem~\ref{Merel}, and call it  the Merel constant.
\end{defn}

We start by proving the following basic lemma.

\begin{lemma}\label{l_part}
	For two abelian groups $A \subseteq B$ and a prime $\ell$, if $B[\ell^{n'}] = A[\ell^{n'}] = A\left[\ell^{n}\right]$ for some non-negative integers $n'$ and $n$ such that $n'>n$, then $B\left[\ell^{\infty}\right] = A\left[\ell^{\infty}\right]$.
\end{lemma}
\begin{proof}
	Since $$
	B\left[\ell^{n}\right]
	= \left(B[\ell^{n'}]\right)\left[\ell^{n}\right]
	= \left(A[\ell^{n'}]\right)\left[\ell^{n}\right]
	\subseteq A\left[\ell^{\infty}\right],
	$$ it is enough to show that $B\left[\ell^{\infty}\right] = B\left[\ell^{n}\right]$. If $B\left[\ell^{\infty}\right] \supsetneq B\left[\ell^{n}\right]$, then there exists $b \in B\left[\ell^{\infty}\right] -B\left[\ell^{n}\right]$ and if $m$ is the smallest non-negative integer $m$ such that $\ell^{m}b=0$ then  $m>n$. Hence, $$
		\ell^{m-n-1} b \in B\left[\ell^{n+1}\right] = \left(B[\ell^{n'}]\right)\left[\ell^{n+1}\right] = \left(A[\ell^{n'}]\right)\left[\ell^{n+1}\right] = \left(A\left[\ell^{n}\right]\right)\left[\ell^{n+1}\right] = A\left[\ell^{n}\right],
	$$ thus, $\ell^{m-1} b=0$, which contradicts the minimality of $m$.
\end{proof}

Now we give a sufficient condition on the extension degree over $K$ over which a given elliptic curve $E/K$ has the growth of $\ell^{\infty}$-torsion subgroups for finitely many primes $\ell$, simultaneously. We restate the following lemma which will be used to establish it.

\begin{lemma}[{\cite[Lemma~3.2]{IK22}}]\label{lem-ours}
	Let $K$ be a number field and $E/K$ be an elliptic curve over $K$. For positive integers $m$, $n$, and $N$ satisfying $m\mid n\mid N$,
$${E\left(K\right)}\left[N\right] \supseteq \bbz/m\bbz \times \bbz/n\bbz \text{ if and only if for some basis } \calB \text{ of } E\left[N\right], G\left(\calB\right) \text{ is contained in }$$
$$		\left\{\begin{pmatrix}a&b\\c&d\end{pmatrix}\in \GL{\bbz/N\bbz}: a\equiv 1,  b\equiv 0 \pmod {m},\text{ and } c\equiv 0, d\equiv 1 \pmod n\right\}.
	$$
\end{lemma}

\begin{prop}\label{small_prime}
	Let $K$ be a number field, $E/K$ an elliptic curve over $K$, and $p$ a prime which is greater than or equal to the maximal prime divisor of the Merel constant~$M\left(K\right)$. Let $d$ be a positive integer whose minimal prime divisor is greater than $p$. Then, for any extension $L/K$ with $[L:K]=d$ and for any prime $\ell \le p$, $$E\left(L\right)\left[\ell^{\infty}\right] = E\left(K\right)\left[\ell^{\infty}\right].$$
\end{prop}
\begin{proof}
	Let $N := M\left(K\right){\displaystyle \prod_{\text{a prime }\ell \le p}}\ell$. The extension degree $\left[K\left(E\left[N\right]\right):K\right]$ divides the order $\left| \GL{\bbz/N\bbz} \right|$. By the Chinese remainder theorem,  $\left| \GL{\bbz/N\bbz} \right| = \prod_{i} \left| \GL{\bbz/\ell_{i}^{e_{i}}\bbz} \right|$ where $\prod_{i} \ell_{i}^{e_{i}}$ is a prime factorization of $N$. For any prime $q$ and any positive integer $e$, the natural projection $\pi : \GL{\bbz/q^{e}\bbz} \to \GL{\bbf_{q}}$ is surjective. Thus, $$
		\left| \GL{\bbz/q^{e}\bbz} \right|
		= \left| \GL{\bbf_{q}} \right| \left| \ker \pi \right|
		= \left(q^{2}-1\right)\left(q^{2}-q\right) q^{4e-4}
		= q^{4e} \left(1-q^{-2}\right)\left(1-q^{-1}\right),
	$$ and so we have that $$
		\left| \GL{\bbz/N\bbz} \right| = N^{4} {\displaystyle \prod_{\text{a prime }\ell \le p}} \left(1-\ell^{-2}\right)\left(1-\ell^{-1}\right).
	$$ We note that every prime divisor of $\left| \GL{\bbz/N\bbz} \right|$ is less than or equal to $p$. Therefore,  $\left[K\left(E\left[N\right]\right):K\right]$ and $\left[L:K\right]$ are relatively prime, so $\Gal\left(L\left(E\left[N\right]\right)/L\right) \cong \Gal\left(K\left(E\left[N\right]\right)/K\right)$ since $K\left(E\left[N\right]\right) \cap L = K$. For any positive divisors $m$ and $n$ of $N$ such that $m \mid n$, we have that $E\left(L\right)\left[N\right] \supseteq \bbz/m\bbz \times \bbz/n\bbz$ if and only if $E\left(K\right)\left[N\right] \supseteq \bbz/m\bbz \times \bbz/n\bbz$ by Lemma~\ref{lem-ours}, i.e., $E\left(L\right)\left[N\right] = E\left(K\right)\left[N\right]$. Then,  this implies that $ E\left(K\right)\left[M\left(K\right)\right]=E\left(K\right)\left[N\right]= E\left(L\right)\left[N\right]$ recalling that the Merel constant $M\left(K\right)$ divides $N$. Since ${E\left(K\right)}_{\tors} \subseteq {E\left(L\right)}_{\tors}$, Lemma~\ref{l_part} completes the proof.
\end{proof}

\subsection{The proofs of our main theorems}


Finally, we prove Theorem~\ref{main} and Theorem~\ref{quad_main}.

\begin{proof}[Proof of Theorem~\ref{main}]
If $\calR\left(K\right)$ in \eqref{setR} is finite, we let $$
	p_{K}= \max \left(\calR\left(K\right)\cup\left\{\text{a  prime } p : p \mid  M\left(K\right)\cdot N_K\right\}\right),
$$ where $N_K$ is given in Proposition~\ref{woRCM_main}.

Let $d$ be a positive integer whose minimal prime divisor is greater than $p_K$ and let $L$ be an extension of $K$ with $[L:K]=d$.

For any prime $\ell \le p_{K}$,  Proposition~\ref{small_prime} implies that $E\left(L\right)\left[\ell^{\infty}\right] = E\left(K\right)\left[\ell^{\infty}\right]$.

For any prime $\ell > p_{K}$,   we note that $p_K\ge 7$, since $7 \mid M(\bbq) \mid M\left(K\right)$ by Mazur's classification of torsion subgroups over $\bbq$ (\cite[Theorem~2]{M78}), so $\ell\geq 11$. 
Also, since $p_{K} \ge 7$, we know that $2,3\nmid [L:K]$. If $\left[K\left(S\right):K\right]$ is divisible by $2$ or $3$ for any point $S \in E\left[\ell\right] -\left\{O\right\}$, then $E\left(L\right)\left[\ell\right] = \left\{O\right\}$  since $2,3\nmid\left[L:K\right]$.
If there exists a point $R \in E\left[\ell\right] -\left\{O\right\}$ such that $2,3 \nmid \left[K\left(R\right):K\right]$, then Proposition~\ref{woRCM_main} implies that there exists a prime $q \in \calR\left(K\right)$ which divides $\left[ K\left(T\right) : K\right]$ for all $T  \in E\left[\ell\right] -\left\{O\right\}$.  Since the minimal prime divisor of $\left[L:K\right]$ is greater than $p_{K}$ and $p_{K} \ge q$ from our choice of $p_K$,  we conclude that $E\left(L\right)\left[\ell\right] = \left\{O\right\}$. So in either case, we have shown that $E\left(K\right)[\ell]=\{O\}=E\left(L\right)[\ell]$ for $\ell>p_K$. Hence, 
$E\left(K\right)_{\tors}=E\left(L\right)_{\tors}$.
\end{proof}

Theorem~\ref{main} implies Theorem~\ref{quad_main}.

\begin{proof}[Proof of Theorem~\ref{quad_main}]
	\cite[Theorem~4]{M95} implies that if $\calK$ is a quadratic field without RCM, the set $\PDI_{2}\left(\calK\right)$ is finite, and so is $\calR\left(\calK\right)$. Hence, it follows from Theorem~\ref{main}.
\end{proof}

Moreover, our results imply Genao's result \cite[Theorem~3]{G22} as well.

\begin{cor}[{\cite[Theorem~3]{G22}}]\label{G22_2}
	Let $K$ be a number field without RCM. Assuming GRH, the answer to Question~\ref{quest_org} is affirmative.
\end{cor}

\begin{proof}
	Under GRH, the set $\PDI_{2}\left(K\right)$ is finite (see \cite[Remark~8]{M95}), and so is $\calR\left(K\right)$. Hence, it follows from Theorem~\ref{main}.
\end{proof}

\begin{remark}
As mentioned in Remark~\ref{PDI-P} and Remark~\ref{PDI-PDI2}, the finiteness of $\PDI_{2}\left(K\right)$ is an essential condition for obtaining our results and concerning Question~\ref{quest_famous}. On the other hand,  we can see  in the proofs of Theorem~\ref{quad_main} and Corollary~\ref{G22_2} that the finiteness of $\PDI_{2}\left(K\right)$ implies the finiteness of $\calR\left(K\right)$. But our final remark is that the converse is unknown to be true.
\end{remark}

\end{document}